\documentclass[reqno]{amsart}
\usepackage{DocumentPreamble-arXiv}

\AddTitle{Are sparse graphs typically determined by their spectrum?}
 
\AddAuthor{Nils Van de Berg\textsuperscript{$\diamond$}}
\email{\href{mailto:n.p.v.d.berg@tue.nl}{n.p.v.d.berg@tue.nl}}
\address{$\diamond$ Eindhoven University of Technology, The Netherlands}
\AddAuthor{Alexander Van Werde\textsuperscript{$\ast$}}
\email{\href{mailto:a.van.werde@uni-muenster.de}{a.van.werde@uni-muenster.de}}
\address{$\ast$ University of Münster, Germany}

\AddKeyword{cospectral graph}
\AddKeyword{giant component}
\AddKeyword{switching method}
\AddKeyword{generalized cospectral}
\AddKeyword{$2$-core}

\AddMSC{05C50}
\AddMSC{05C80}
\AddMSC{15B52}
\AddMSC{60B20}
\begin{document}
\maketitle
\begin{abstract}
    We investigate whether it is typical for a sparse graph to be uniquely characterized by its adjacency spectrum up to isomorphism.
    Our first result shows that the giant component of an Erd\H{o}s--R\'enyi graph is cospectral  when the average degree is sufficiently small.  
    The proof relies on the existence of a specific pendant tree, combined with a method by Schwenk that swaps trees to construct a cospectral mate.  
    
    It seems possible that pendant trees are essentially the only obstruction, meaning that the giant should become characterized by spectrum with high probability if one prunes these by considering the $2$-core.   
    The majority of the paper is devoted to theoretical and numerical evidence supporting this concept.  
    Our main theorem in this direction establishes that local switching methods can not cause the $2$-core to be cospectral.
    We also discuss $\mathbb{R}$-cospectrality and rational cospectrality at fixed level. 
\end{abstract}

\section{Introduction}
A fundamental problem in spectral graph theory is whether the eigenvalues of the adjacency matrix characterize a given graph up to isomorphism.
This is known for some special classes like cycles or disjoint unions of complete graphs; see \eg \cite{van2003graphs}. 
Conversely, there are also examples of non-isomorphic graphs with equal spectra, called \emph{cospectral mates}. 

One may wonder if such examples are typical or exceptional. 
A striking result by Schwenk from 1973 established that almost all trees admit a cospectral mate \cite{schwenk1973almost}, thus showing that cospectrality is typical for trees.
This naturally also raised the question what might be true for generic graphs, which turned out to be a challenging problem. 
For a long time, even the expected direction of the answer was unclear. 
For instance, Godsil wrote in his 1993 book that ``It is still an open question whether almost all graphs are characterized by their characteristic polynomial. It is not even clear if we should seek to prove this, or disprove this.'' \cite[p73]{godsil1993algebraic}.

Evidence on the direction of the answer came from an exhaustive search through graphs on $n\leq 11$ vertices by Haemers and Spence \cite{haemers2004enumeration} which suggested that a large fraction of cospectral graphs arose from a specific local construction of cospectral graphs known as \emph{Godsil--McKay switching} \cite{godsil1982constructing}. 
That specific construction can be proved to only produce an asymptotically negligible fraction of graphs, leading to the belief that the opposite of Schwenk's result may happen. 
\emph{Haemers' conjecture} asserts that almost all graphs are characterized by their spectrum \cite{van2003graphs,haemers2016almost}. 
A proof remains elusive, but further evidence has recently been found by Wang and Wang, who estimated via Monte Carlo sampling that at least $99\%$ of graphs on $50$ vertices are characterized if the spectrum of the complement graph is also given \cite{wang2025haemers}.

\pagebreak[4]
In probabilistic language, Haemers' conjecture means that an Erd\H{o}s--R\'enyi graph from the $G(n,p)$ model with edge probability $p = 1/2$ is characterized by its spectrum with probability tending to one as $n\to \infty$.
Such a graph is rather dense, as the average degree is linear in $n$. 
Thus, Haemers' conjecture can be interpreted as stating that \emph{dense} graphs are typically characterized by their spectrum. 
The present paper aims to determine what should be typical in sparse settings, which we do by considering the giant component in a regime where $p= \lambda/n$ for $1 < \lambda \ll n$.

We find in \Cref{thm: GiantCospectral} that the giant is cospectral if the average degree $\lambda$ is sufficiently small.
This also implies that the full graph is cospectral for small $\lambda$, as a disconnected graph is cospectral whenever any of its components is so.
The proof of the giant's cospectrality relies on specific pendant trees as an obstruction to spectral characterization. 
The majority of this paper then investigates what is typical for sparse graphs without pendant trees, which we do by considering the $2$-core of the giant component. 
Our main result is Theorem~\ref{thm: NoSwitching} and shows that every possible local switching method fails on the $2$-core. 
In particular, this rules out Godsil--McKay switching of bounded size but the claim is substantially more general than only that particular method.    
We discuss additional theoretical evidence for the resulting conjecture that the $2$-core should be characterized by its spectrum in \Cref{sec: EvidenceOpen}. 
(Numerical evidence is given in \Cref{sec: NumericalEvidence}.)  
The structure of the remainder of the paper is outlined in \Cref{sec: Outline}.   

\subsection{The giant is cospectral}\label{sec: GiantIsCospectral}
Recall that two non-isomorphic graphs are \emph{cospectral} if they have the same adjacency spectrum. 
Two cospectral graphs are called \emph{$\bbR$-cospectral} if the complement graphs are also cospectral.
This generalization dates back to the 1980 work of Johnson and Newman \cite{johnson1980note} which established various equivalent formulations.
Aside from its independent interest, the relevance for the current paper is that this notion will enable numerical investigations.  

\begin{theorem}\label{thm: GiantCospectral}
    Consider a $G(n,p)$ random graph with $p = \lambda/n$ where $\lambda = \lambda_n >1$ remains bounded away from one as $n\to \infty$. 
    Further, consider an arbitrary sequence $h(n)$ with $h(n)\to \infty$. 
    Then,  
    \begin{enumerate}[leftmargin = 2em]
        \item\label{item: Cospectral} The largest component $C_{\textnormal{giant}}$ admits a cospectral mate with probability tending to one if 
        \begin{equation} 
            \lambda \leq  \ln(n)/8 +  \ln(\ln(n)) - h(n).\label{eq: Cospectral_lambda_condition} 
        \end{equation}
        \item\label{item: R-Cospectral} The largest component $C_{\textnormal{giant}}$ admits a $\bbR$-cospectral mate with probability tending to one if 
        \begin{equation} 
            \lambda \leq\ln(n)/10 +  \ln(\ln(n)) - h(n). \label{eq: R-Cospectral_lambda_condition}
        \end{equation}
    \end{enumerate}
\end{theorem}

The proof is given in \Cref{sec: ProofGiantCospectral}.
The idea is visualized in \Cref{fig: GiantZoom} and is as follows.
It is known by Schwenk \cite{schwenk1973almost} that there exist rooted trees $T_1,T_2$ on $9$ nodes such that for any graph $G$, the graph found by joining $T_1$ to $G$ at a vertex is cospectral to the graph that is found if one instead joins with $T_2$ at that vertex.
Hence, any graph with a pendant copy of $T_1$ has a cospectral mate. 
The main content in the proof of  \Cref{thm: GiantCospectral} is to determine for what range of $\lambda$ a pendant copy of $T_1$ exists in the giant, which is where the bound in \eqref{eq: Cospectral_lambda_condition} arises. 
The condition \eqref{eq: R-Cospectral_lambda_condition} follows similarly using trees with $11$ nodes that yield $\bbR$-cospectrality, due to Godsil and McKay \cite{godsil1976some}.

\begin{figure}[hb]
    \centering
    \includegraphics[width=0.94\linewidth]{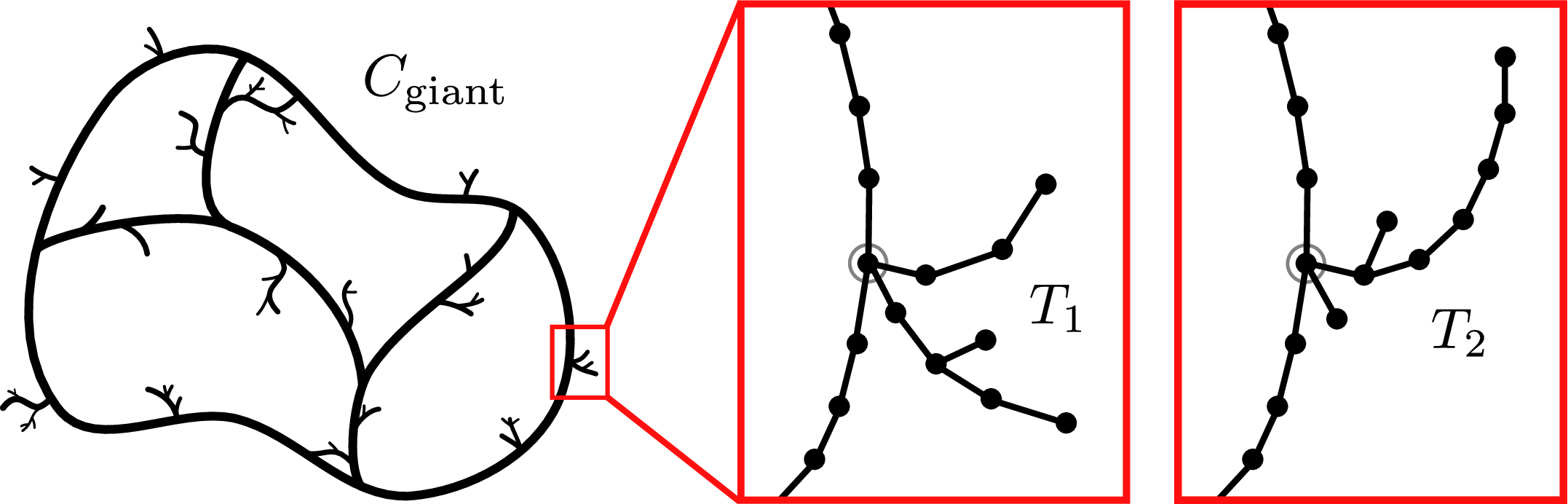}
    \caption{
    The giant can be decomposed as a skeleton consisting of cycles with pendant trees attached to some nodes.
    (See \cite{ding2014anatomy} for precise results of this nature.)
    To show that it is cospectral, it suffices to find a specific rooted tree $T_1$ among these pendant trees. 
    A non-isomorphic cospectral graph is then found by replacing $T_1$ with a different rooted tree $T_2$. 
    }
    \label{fig: GiantZoom}
\end{figure}

Swapping pendant trees can also be used to derive lower bounds on the number of cospectral mates. 
The following result illustrates this principle by showing that there are exponentially many cospectral mates when the average degree is fixed:  
\begin{proposition}\label{prop: ExponentiallyManyCospectral}
    For every fixed $\lambda >1$ there exists $c>0$ such that $C_{\textnormal{giant}}$ has at least $\exp(cn)$ non-isomorphic $\bbR$-cospectral mates with probability tending to one. 
\end{proposition}

Here, and in all subsequent results, it should be understood that we use the same notation as in \Cref{thm: GiantCospectral}. 
That is, $C_{\textnormal{giant}}$ refers to the largest component in a $G(n,p)$ graph with $p=\lambda/n$. 

The main point that requires some care in the proof of \Cref{prop: ExponentiallyManyCospectral} is that automorphisms of the graph could result in different swaps of pendant trees giving isomorphic graphs. 
However, this is ultimately insubstantial because the number of vertices with a nontrivial orbit can be controlled using a result by Verbitsky and Zhukovskii \cite{verbitsky2024canonical}.
Details can be found in \Cref{sec: ProofGiantCospectral}.

\subsection{Are pendant trees the only obstruction?}
Theorem \ref{thm: GiantCospectral} and Proposition \ref{prop: ExponentiallyManyCospectral} show that sparse graphs are typically cospectral because they contain pendant trees. 
We would like to propose that this might be the only reason for cospectrality in this setting. 
One precise way to phrase this is to consider the \emph{$2$-core}, denoted $\operatorname{Core}_2(C_{\textnormal{giant}})$, which is the graph found by removing all pendant trees. 
More generally, the \emph{$k$-core} for $k\geq 2$ is the subgraph that remains when one iteratively removes all vertices of degree $<k$.  
\begin{conjecture}\label{conj: 2Core}
    For every $1 <\lambda \leq n/2$ that remains bounded away from one as $n\to \infty$, it holds that the $2$-core of the giant is characterized by its spectrum with probability tending to one. 
\end{conjecture}
A closely related philosophy that local structures should be the typical obstruction for spectral properties has been crucial in the study of the rank of discrete random matrices \cite{tikhomirov2020singularity,costello2008rank,costello2010rank,addario2014hitting}.
Particularly closely related is work of Feber, Glasgow, Kwan, Sah, and Sawhney \cite{ferber2023singularity,glasgow2025exact}, who investigated the $k$-core of an Erd\H{o}s--R\'enyi graph in the regime with constant average degree.
The $2$-core is an edge case in that theory as local dependencies for the rank then still occur with nonzero probability, while the $k$-core with $k\geq 3$ ensures that all local dependencies are removed. 
Correspondingly, it was shown in \cite{ferber2023singularity} that the adjacency matrix of the $k$-core is nonsingular for $k\geq 3$, while \cite{glasgow2025exact} showed that the $2$-core is singular with probability strictly between $0$ and $1$.

Given that the $2$-core still allows local obstructions causing singularity with nonvanishing probability, one may wonder whether a similar issue applies for cospectrality.
Aside from swapping pending trees, other methods to produce cospectral graphs that rely on switching connections on a small subset of nodes are indeed known.
As mentioned in the introduction,  Godsil--McKay switching \cite{godsil1982constructing} is believed to account for a large fraction of cospectral graphs in the dense setting of Haemers' conjecture. 
Further, other examples such as \emph{Abiad--Haemers switching} \cite{abiad2012cospectral} or \emph{Wang--Qiu--Hu switching} \cite{wang2019cospectral} are also known. 
The number of cospectral graphs that can be produced using such methods was recently analysed by Abiad, Van de Berg, and Simoens \cite{abiad2025counting}. Based on their unifying framework we consider the following notion of local switching:
\begin{definition}\label{def: SwitchingOfSizeM}
    Fix an integer $m\geq 1$ and a graph $G$ with $n\geq m$ vertices. 
    Then, a graph $H$ is said to be obtained by a \emph{switching of size $\leq m$} if there exists an ordering of the vertices of $G$ and $H$ such that their adjacency matrices satisfy
    $
        A_H = (Q \oplus I_{n-m})^TA_G (Q\oplus I_{n-m})
    $
    for some $m\times m$ orthogonal matrix $Q$ that is not a permutation. 
    Here, $I_{n-m}$ is the identity matrix.
\end{definition}

\begin{example}[Godsil--McKay switching {\cite{godsil1982constructing}}]\label{ex: GosilMcKay}
    Suppose that $m=2k$ is an even integer. 
    Assume that $G$ admits a set of vertices $X = \{x_1,\ldots,x_m \}$ such that the induced subgraph $G[X]$ is regular and such that every $v\not\in X$ has either $0,k$ or $2k$ neighbors in $X$.
    Then, a graph cospectral to $G$ can be found by swapping the adjacency to $X$ for every $v\not\in X$ with exactly $k$ neighbors in $X$, thus making it adjacent to only the other $k$ vertices in $X$. 
    The associated orthogonal matrix is 
    \begin{equation}
        Q =
        \frac{1}{k}\begin{pmatrix}
            J_k - kI_k &  J_k \\ 
            J_k & J_k - kI_k 
        \end{pmatrix}, 
    \end{equation}
    where $J_k$ is the $k\times k$ all-ones matrix. 
    Let us note that it can sometimes occur that the graph constructed by the switching is isomorphic to $G$. 
\end{example}
\pagebreak[3]
\begin{example}[Swapping pending graphs]\label{ex: Pending} 
    Consider cospectral graphs $H_1,H_2$ on $m+1$ vertices for which there exist $x_1 \in H_1$ and $x_2 \in H_2$ such that $H_1\setminus \{x_1 \}$ is also cospectral to $H_2 \setminus \{x_2 \}$.
    Then, for any graph $G$, the graph $G\cdot H_1$ found adding an $x_1$-rooted copy of $H_1$ at some vertex is cospectral to the graph $G \cdot H_2$ found by instead adding $H_2$ at that vertex; see \eg \cite[p.59 \& p.65]{godsil1993algebraic}.   
    In fact, a result by Farrugia  \cite[Theorem 10]{farrugia2019overgraphs} yields an $m\times m$ orthogonal matrix $Q$ with
    \begin{equation}
        Q^{\T} A_{1} Q = A_{2} \quad \textnormal{ and }Q^{\T} v_1 = v_2, 
    \end{equation}
    where $A_i$ is the adjacency matrix of $H_i \setminus \{x_i \}$ and $v_i\in \{0,1 \}^{H_i \setminus \{x_i \}}$  is the indicator vector of the neighbors of $x_i$ in $H_i$.  
    It follows that the cospectrality of $G\cdot H_1$ and $G\cdot H_2$ can be realized by a switching method of size $\leq m$ applied to $G\cdot H_1$.
    In particular, this allows interpreting the swaps of specific pendant trees from \cite{schwenk1973almost,godsil1976some} used in \Cref{sec: GiantIsCospectral} as a switching method. 
\end{example}

Note that two symmetric matrices share the same spectrum if and only if they can be transformed into each other by an orthogonal matrix. 
It follows that two graphs are cospectral if and only if there exists a switching that relates them, but the size of the switching can potentially be arbitrarily large.
A switching of size $m$ will however only modify $G$ on a subgraph of size $m$ and on the edges connecting that subgraph to its complement.
The size of the switching hence gives a way to quantify whether cospectrality occurs for a local reason or if it necessitates a global conspiracy in the large-scale graph structure.
The following \Cref{thm: NoSwitching} shows that such local obstructions do not produce cospectrality in the $2$-core,  thus giving evidence for \Cref{conj: 2Core}.
\begin{theorem}\label{thm: NoSwitching} 
    Suppose that $\lambda>1$ remains bounded away from one as $n\to \infty$. 
    Fix some integer $m\geq 1$ and assume that $\lambda = o(n^{1/(2m+2)})$. 
    Then, with probability tending to one, every graph $H$ obtainable from the giant's $2$-core by a switching of size at most $m$ is isomorphic to it.
    That is,
    \begin{equation}
        \lim_{n\to \infty}\bbP\bigl(\operatorname{Core}_2(C_{\textnormal{giant}})\textnormal{ admits a switching of size } \leq m \textnormal{ that yields a non-isomorphic graph}  \bigr) =0.\nonumber
    \end{equation}
\end{theorem}

The proof of \Cref{thm: NoSwitching} is given in \Cref{sec: SwitchingFails}. 
We exploit that a switching can only be applied in a connected graph of minimum degree $2$ if there are some small cycles. 
The $2$-core does not have many of such small cycles, but a few can still occur.
It is therefore possible that a switching is applicable. We show that the result of the switching will be isomorphic to the original graph as long as there are no cycles near each other; see \Cref{fig: GiantZoom2}. 
The assumption that $\lambda = o(n^{1/(2m+2)})$ is here used to rule out the collision of two cycles.

To prove that switchings result in isomorphic graphs if they do not act on connected cycles, we exploit that switching methods also induce cospectral pairs on certain subgraphs. 
We use that cycles are determined by their spectrum and that the cospectrality of their subgraphs are understood to show that the induced cospectral subgraphs must be isomorphic. 
The main difficulty appears in showing that these isomorphisms of induced subgraphs also extend to isomorphisms of the whole $2$-core; see the proof of \Cref{lem:switchingcyclesiso}.  

\begin{figure}[hb]
    \centering
    \includegraphics[width=0.94\linewidth]{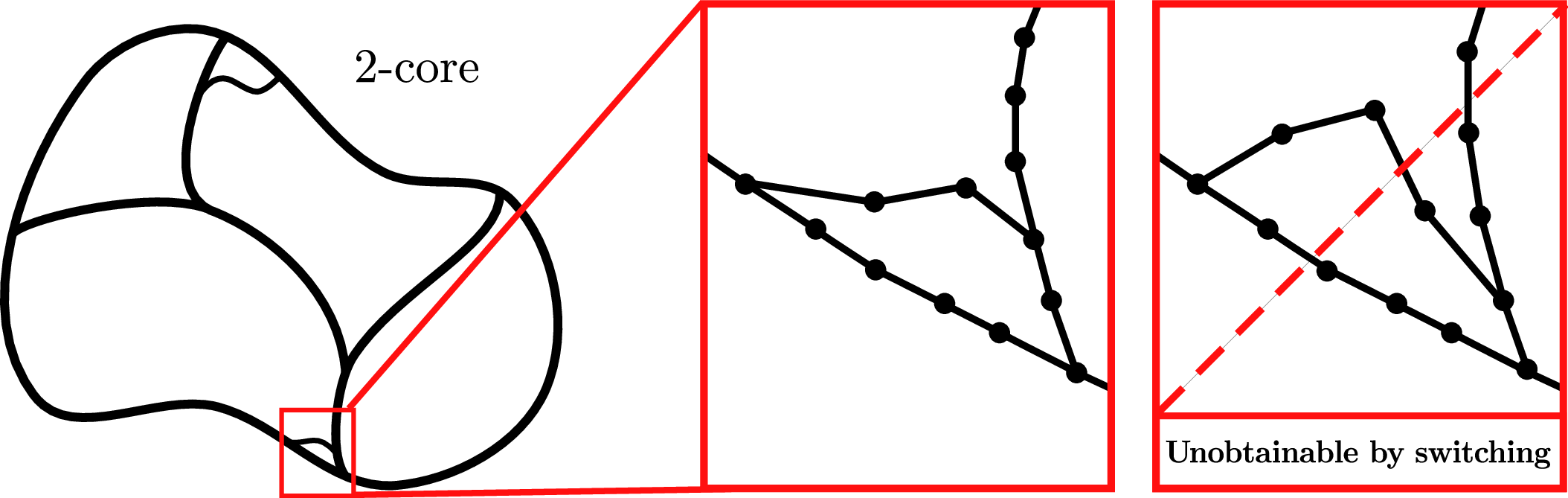}
    \caption{
    The $2$-core of the giant includes a few small cycles with nonzero probability. 
    It is possible for a switching method to act on such a cycle, but we show in the proof of \Cref{thm: GiantCospectral} that it is impossible to obtain a non-isomorphic graph in this fashion.
    }
    \label{fig: GiantZoom2}
\end{figure}

\pagebreak[3]
This new proof strategy also has potential for other graph classes than sparse connected graphs with minimum degree $2$.
Specifically, the benefit of the strategy is that it suffices to understand the cospectrality properties of only the special class of graphs that can occur as small subgraphs, as we do in \Cref{sec: GenCospec}.
One can then leverage this understanding to build a global isomorphism for switching methods. 
A potential bottleneck is that understanding the cospectrality properties for even special classes of graphs can be nontrivial.
Some classes where such properties are relatively well-understood are complete graphs, complete bipartite graphs and graphs with spectral radius at most $2$, see \eg \cite[Section 5]{van2003graphs} and \cite{Cvetkovic1975}.

\subsection{Further theoretical evidence and open questions}\label{sec: EvidenceOpen} 
A proof of \Cref{conj: 2Core} appears out of reach at the moment, considering that Haemers' conjecture is a special case. 
More insight on the sparse setting could however be found by studying other relaxations.

For example, can one prove that any modification of the $2$-core that alters only a bounded number of edges does not yield a non-isomorphic cospectral graph? 
This is not implied by \Cref{thm: NoSwitching}. 
(For instance, because switching must replace the modified subgraph by a cospectral graph.) 
The conceptual difference is more or less that switching only allows modifications that are guaranteed to be cospectral, meaning that the reason for cospectrality must be local, while arbitrary edge modifications could potentially also be cospectral through a global coincidence.   

Another relaxation of classical cospectrality arises by imposing that the orthogonal matrix has rational entries. 
Two graphs $G,H$ are \emph{cospectral through a rational matrix} if there exists orthogonal $Q$ with rational entries such that $A_H = Q^{\T} A_G Q$. 
The least integer $\ell\geq 1$ with $\ell Q \in \bbZ^{n\times n}$ is the \emph{level of $Q$}. 
The rational notion is related to that of $\bbR$-cospectrality in \Cref{sec: GiantIsCospectral}. 
Indeed, it is known for \emph{controllable graphs} that $\bbR$-cospectrality is equivalent to the existence of a rational orthogonal matrix $Q$ with $A_G = Q^{\T} A_G Q$ and $Q e = e$ where $e= (1,\ldots,1)^{\T}$ is the all-ones vector due to Wang and Xu \cite[Lemma 2.4]{wang2006sufficient}, and dense Erd\H{o}s--R\'enyi graphs from the $G(n,p)$ model with $p=1/2$ are known to be controllable due to O'Rourke and Touri \cite[Theorem 1.4]{o2016conjecture}.

Wang and Zhao \cite{wang2025graphscospectralmatefixed} recently proved that cospectrality at a fixed level $\ell$ does not occur with high probability in the dense regime with fixed edge probability $p \in (0,1)$. 
In fact, inspecting the proofs in \cite{wang2025graphscospectralmatefixed} shows that they also extend to a sparser regime with logarithmic average degree: 
\begin{theorem}[Wang and Zhao {\cite{wang2025graphscospectralmatefixed}}]\label{thm: WangLevel}
    For every fixed $\ell \geq 2$ there exists a constant $C>0$ such that the following holds for an Erd\H{o}s--R\'enyi random graph with $p = \lambda/n$ if $C\ln(n) \leq \lambda \leq n/2$: 
    \begin{equation}
        \lim_{n\to \infty}\bbP\bigl(G(n,p) \textnormal{ is cospectral through a rational matrix at level } \ell  \bigr) = 0.
    \end{equation}
\end{theorem}
\begin{proof}
    It is shown in \cite[eq.28]{wang2025graphscospectralmatefixed} that the probability that a $G(n,p)$ graph has a cospectral mate of level $\ell'$ with $\ell' \mid \ell$ is at most $\epsilon_n^2/(1-\epsilon_n)$ for all large $n$ where 
    $
        \epsilon_n \leq n^{c_1} \hat{p}^{c_2n }
    $
    for certain constants $c_1,c_2 >0$ depending on $\ell$ and with $\hat{p} \de  \max\{p,1-p \}$. 
    Using that $\hat{p} = 1-\lambda/n$  now ensures that $\epsilon_n\to 0$ by taking the constant $C$ in the lower bound $C\ln(n) \leq \lambda$ sufficiently large. 
    \end{proof} 
The condition $\ell \geq 2$ is necessary to rule out the trivial case where $Q$ is a permutation matrix that occurs if $\ell=1$. 
\Cref{thm: WangLevel} hence states that it is impossible to have a nontrivial rational cospectrality with a small denominator if the average degree is sufficiently large, thus giving evidence in favor of \Cref{conj: 2Core}. 
Indeed, if the average degree is logarithmic with a large constant, then the giant and its $2$-core both coincide with the full graph. 
This complements the regime with $\lambda$ sufficiently small that was considered in \Cref{thm: NoSwitching}.

Further recall that \Cref{thm: GiantCospectral} yields cospectrality for the giant component if the average degree is logarithmic with an explicit constant. 
Taken together with \Cref{thm: WangLevel} this suggests that logarithmic average degree is a critical scale of sparsity where the graph transitions from almost always being cospectral to almost never cospectral. 
One is led to wonder whether a sharp phase transition occurs for the cospectrality or $\bbR$-cospectrality of the giant as soon as the average degree exceeds that in \eqref{eq: Cospectral_lambda_condition} or \eqref{eq: R-Cospectral_lambda_condition}, respectively. 
It would be interesting future work if one could find evidence to clarify if the constants appearing there should be sharp. 

\subsection{Outline}\label{sec: Outline}
The proofs of \Cref{thm: GiantCospectral} and \Cref{prop: ExponentiallyManyCospectral} are given in \Cref{sec: ProofGiantCospectral}. 
The proof of \Cref{thm: NoSwitching} is given in \Cref{sec: SwitchingFails}.  
We finally consider numerical evidence for a variant of \Cref{conj: 2Core} with $\bbR$-cospectrality in \Cref{sec: NumericalEvidence} using an algorithm by Wang and Wang \cite{wang2025haemers} with some additional optimizations to enable larger graphs.

\section{The giant is cospectral --- Proofs of \texorpdfstring{\Cref{thm: GiantCospectral}}{Theorem} and \texorpdfstring{\Cref{prop: ExponentiallyManyCospectral}}{Proposition}}\label{sec: ProofGiantCospectral}
Recall from \Cref{ex: Pending} that for rooted graphs $H,K$, we denote $H\cdot K$ for the graph found by uniting the graphs and identifying the roots. 
A graph of the form $G \cong H\cdot T$ for some rooted tree $T$ is said to have a \emph{pendant copy of $T$}. 
(Recall also \Cref{fig: GiantZoom}.)
The main content in the present section is to determine for what range of $\lambda$ there exists a pendant copy of a tree $T$ on $t$ vertices in the giant.  
The desired cospectrality results then follow by using trees from \cite{schwenk1973almost,godsil1976some}. 

We use a second moment calculation. 
Identify the vertices of $T$ with $\{1,\ldots,t \}$ such that vertex $t$ refers to the root.     
Further, given vertices $v_1,\ldots, v_t \in \{1,\ldots,n \}$, let $\kT(v_1,\ldots,v_t)$ denote the event where the original graph can be decomposed as $G(n,p) = H\cdot T$ with vertices $1,\ldots,t$ of $T$ being mapped to $v_1,\ldots,v_t$, respectively. 
We count pendant copies of $T$ in the giant component: 
\begin{equation}
        N \de  \sum_{v_1\in \{1,\ldots,n \}} \sum_{v_2\neq v_1} \cdots \sum_{v_t \not\in \{v_1,\ldots,v_{t-1} \}} \bb1\bigl\{ \kT(v_1,\ldots,v_t)  \textnormal{ and }v_t \in  C_{\textnormal{giant}}  \bigr\}.  \label{eq:NormalBeetle}
    \end{equation}
We determine $\bbE[N]$ and $\bbE[N^2]$ in the following \Cref{lem: EN,lem: EN2}, respectively. 

\begin{lemma}\label{lem: EN}
    Fix a rooted tree $T$ on $t\geq 2$ vertices. 
    Assume that $\lambda>1$ remains bounded away from one and that $p = \lambda/n$ tends to zero as $n\to \infty$. 
    Then, with $c_{\lambda} \de  \bbE[\#C_{\textnormal{giant}}/n]$ and $N$ as in \eqref{eq:NormalBeetle},     
    \begin{equation}
        \bbE[N] = \bigl(1-o(1)\bigr) n^t c_{\lambda}  p^{t-1} (1-p)^{(t-1)n}.\label{eq:WittyBall} 
    \end{equation} 
    In particular, $\bbE[N]\to \infty$ if and only if there exists a sequence $h(n)$ with $h(n)\to \infty$ such that 
    \begin{equation} 
        1< \lambda \leq \ln(n)/(t -1) +  \ln(\ln(n)) - h(n).  \label{eq:UnripeSong}
    \end{equation} 
\end{lemma}
\begin{proof}
    The event $\kT(v_1,\ldots,v_t)$ means that  $v_i \sim v_j$ if and only if $\{i,j \}\in T$, and that there are no edges from the $v_i$ with $i\leq t-1$ to external vertices $v \not\in \{v_1,\ldots,v_{t}\}$.
    Here, the tree $T$ has exactly $t-1$ edges and $\binom{t}{2}-(t-1)$ non-edges. 
    Further, there are exactly $(n-t)(t-1)$ potential edges from the $v_i$ with $i\leq t-1$ to the complement of $\{v_1,\ldots, v_t\}$. 
    Hence, using that edges in the $G(n,p)$ model are independent, 
    \begin{equation} 
        \bbP\bigl(\kT(v_1,\ldots,v_t)\bigr)  = p^{t-1}  (1-p)^{(t-1)(n-t) + \binom{t}{2} - (t-1)}  = \bigl(1-o(1)\bigr) p^{t-1} (1-p)^{(t-1) n},    \label{eq:TinyPen}
    \end{equation}
    where the second step used that $1-p = 1 - o(1)$.
    Conditional on $\kT(v_1,\ldots,v_t)$, the graph induced by $\{1,\ldots,n \} \setminus \{v_1,\ldots,v_{t-1} \}$ is again an Erd\H{o}s--R\'enyi graph with parameter $p$.
    Hence, since $\lambda >1$ remains bounded away from one, it holds with high probability that the latter graph again has a unique giant component containing $(1-o(1)) c_{\lambda}n$ nodes with $c_{\lambda} = \bbE[ \#C_{\textnormal{giant}}/n]$ bounded away from zero. 
    In particular, the probability that $v_t$ belongs to this component is
    \begin{equation} 
        \bbP\bigl(v_t \in C_{\textnormal{giant}} \mid \kT(v_1,\ldots,v_t)\bigr) = (1-o(1))c_{\lambda}.\label{eq:WeepyNet} 
    \end{equation} 
    Using that there are $n(n-1)\cdots (n-t+1)=(1-o(1))n^t$ summands in \eqref{eq:NormalBeetle} now yields \eqref{eq:WittyBall}. 

    That \eqref{eq:UnripeSong} is necessary and sufficient to ensure that $\bbE[N]\to \infty$ now follows by direct calculation.
    Note that $(1- \lambda^2 /n )\exp(-\lambda) \leq (1-\lambda/n)^{n} \leq \exp(-\lambda)$ for $\lambda \leq n$ \cite[p.266]{mitrinovic1970analytic}. 
    The upper bound implies that the right-hand side of \eqref{eq:UnripeSong} tends to zero if $\lambda = \omega(\ln(n))$. 
    It further follows that $(1-\lambda/n)^{(t-1)n} = (1-o(1))\exp(-\lambda (t-1 ))$ if $\lambda = o(\sqrt{n})$. 
    Hence, using that that $n^t p^{t-1} =   n \lambda^{t-1}$  shows that \eqref{eq:WittyBall} tends to infinity at rate $n^{1 -o(1)}$ if $\lambda =o(\ln(n))$, and has order $\exp((t-1)h(n))$ if $\lambda = \Theta(\ln(n))$ with $\lambda = \ln(n)/(t -1) +  \ln(\ln(n)) - h(n)$.  
\end{proof}

\begin{remark}
    If $h(n)$ is replaced by a constant in \eqref{eq:UnripeSong}, then \Cref{lem: EN} shows that $\bbE[N]$ is of a constant order and a more complicated proof would show that $N$ has Poissonian statistics. 
    In particular, $N$ is then zero with nonzero probability. 
    The bound on $\lambda$ in \Cref{cor: PendantTree} is therefore necessary to ensure that $\lim_{n\to \infty}\bbP(N>0)=1$. The rest of the second moment calculation will show that it is also sufficient. 
\end{remark}

\begin{lemma}\label{lem: SharedRoot}
    Fix a rooted tree $T$ on $t\geq 2$ vertices. 
    Consider a connected graph $G$ with at least $2t$ vertices. 
    Then, every pair of pendant copies of $T$ in $G$ with shared vertices must share roots.    
\end{lemma} 
\begin{proof}
    Assume that we are given sets of vertices with $\{v_1,\ldots,v_t \} \cap \{w_1,\ldots,w_t \} \neq \emptyset$ such that $\{v_1,\ldots,v_t \}$ and $\{w_1,\ldots,w_t \}$ both yield pendant copies of $T$ in $G$ with roots $v_t$ and $w_t$, respectively.  
    The claim is that it must hold that $v_t = w_t$.
    Suppose to the contrary that $v_t\neq w_t$.

    That $G$ is connected with at least $2t > t$ vertices implies that $\{v_1,\ldots,v_t \} \cap \{w_1,\ldots,w_t\}$ has at least one edge leading to its complement. 
    Assume that there is such an edge to the complement of $\{v_1,\ldots,v_t \}$, the other case proceeds identically. 
    Then, since the definition of pendant trees yields that $v_t$ is the only vertex in $\{v_1,\ldots,v_t \}$ that may connect to external vertices, we have $v_t \in \{w_1,\ldots,w_t \}$ and even $v_t \in \{w_{1},\ldots,w_{t-1} \}$ since $v_t \neq w_t$. 
    Now consider the neighborhoods:
    \begin{equation}
        N_0 \de \{v_t \}, \qquad  N_{i+1} \de N_i \cup \bigl\{v\in \{v_1,\ldots,v_t \}: \exists v'\in N_i,\, v\sim v' \bigr\}. \label{eq:GladEel}
    \end{equation}
    Then, using that $w_t$ is the only vertex in $\{w_1,\ldots,w_t \}$ with external neighbors, it holds that  
    \begin{equation}
        N_{i}\subseteq \{w_1,\ldots,w_{t-1}\} \implies N_{i+1}\subseteq\{w_1,\ldots,w_t \}.\label{eq:VividParrot}
    \end{equation}  
    Note that $N_0 = \{v_t \} \subseteq \{w_1,\ldots,w_{t-1} \}$, but it is not possible that $N_i \subseteq  \{w_1,\ldots,w_{t-1}\}$ for all $i$ since the connectivity of $T$ implies that $N_{i} = \{v_1,\ldots,v_t \}$ has cardinality $t$ for large $i$. 
    Consequently, \eqref{eq:VividParrot} implies that there exists $i$ with $w_{t}\in N_i$. 

    The definition \eqref{eq:GladEel} yields that $N_i\subseteq \{v_1,\ldots,v_t \}$.  
    This implies that $w_t \in \{v_1,\ldots,v_{t} \}$ and hence 
    $
        w_t \in \{v_1,\ldots,v_{t-1} \}
    $ 
    since $v_t \neq w_t$. 
    Now, also using that $v_{t} \in \{w_{1},\ldots,w_{t-1} \}$, it follows that there are no edges from $\{v_1,\ldots,v_t \}\cup \{w_1,\ldots,w_t \}$ to its complement. 
    This contradicts the assumption that $G$ is connected with at least $2t$ vertices since $\# \{v_1,\ldots,v_t \}\cup \{w_1,\ldots,w_t \} = 2t - \#\{v_1,\ldots,v_t \}\cap \{w_1,\ldots,w_t \} < 2t$.
    We conclude that $v_t = w_t$. 
\end{proof}
\begin{lemma}\label{lem: EN2}
     Fix a rooted tree $T$ on $t\geq 2$ vertices. 
     Assume that $\lambda>1$ remains bounded away from one and that there exists a sequence $h(n)$ with $h(n)\to \infty$ such that \eqref{eq:UnripeSong} holds.
     Then,  
     \begin{equation}
        \bbE[N^2] = (1-o(1)) \bbE[N]^2.\label{eq:DryUnit} 
     \end{equation}
\end{lemma}
\begin{proof}
    It follows similarly to \eqref{eq:TinyPen}--\eqref{eq:WeepyNet} that for every pair of sets of vertices satisfying $\{v_1,\ldots,v_t \}\cap \{w_1,\ldots,w_t \} = \emptyset$ that with $c_{\lambda} = \bbE[\#C_{\textnormal{giant}}/n]$, 
    \begin{equation} 
        \bbP(\kT(v_1,\ldots,v_t), \kT(w_1,\ldots,w_t), v_t\in C_{\textnormal{giant}},  w_t \in C_{\textnormal{giant}})= \bigl(1 - o(1)\bigr) \bigl(c_{\lambda}  p^{t-1} (1-p)^{(t-1) n} \bigr)^2.    
    \end{equation}    
   The double sum that results from substituting the definition \eqref{eq:NormalBeetle} of $N$ in $\bbE[N^2]$ has $(1-o(1)) n^{2t}$ terms indexed by vertex sets with $\{v_1,\ldots,v_t \}\cap \{w_1,\ldots,w_t \} = \emptyset$. 
   Thus, a contribution of order $(1-o(1))(n^t c_{\lambda}  p^{t-1} (1-p)^{(t-1)n})^2 = (1-o(1)) \bbE[N]^2$ arises from such terms.
    Consequently, 
    \begin{align} 
        \bbE\bigl[N^2\bigr] ={}&{} \bigl(1-o(1)\bigr)\bbE\bigl[N\bigr]^2\label{eq:JollyRug} \\ 
        &\  +   \sum_{\{v_1,\ldots,v_t \} \cap \{w_1,\ldots,w_t \} \neq \emptyset} \negquad  \bbP\bigl(\kT(v_1,\ldots,v_t), \kT(w_1,\ldots,w_t), v_t\in C_{\textnormal{giant}}, w_t \in C_{\textnormal{giant}}\bigr).\nonumber   
    \end{align}   
    It remains to bound the sum on the right-hand side of \eqref{eq:JollyRug}.     
    
    \Cref{lem: SharedRoot} ensures that all terms in \eqref{eq:JollyRug} with $v_t \neq w_t$ are zero.  
    Further, for every $i\leq t-1$, there are $O(n^{2t-i-1})$ terms with $v_t = w_t$ and $\#\{ v_1,\ldots,v_{t-1}\}\cap \{w_1,\ldots w_{t-1} \} = i$.
    It is only possible for both $\kT(v_1,\ldots,v_t)$ and $\kT(w_1,\ldots,w_t)$ to occur if $\{v_1,\ldots,v_t \}\cup \{w_1,\ldots,w_t \}$ induces a tree on $2t - i -1$ vertices with root $v_t = w_t$, and all vertices except $v_t = w_t$ do not have external connections.  
    It hence follows similarly to  \eqref{eq:TinyPen}--\eqref{eq:WeepyNet} that each such term in \eqref{eq:JollyRug} has magnitude $O(p^{2t -i - 2} (1-p)^{(2t - i -2)n})$.
    Thus, using that $n^{2t-i-1} p^{2t-i-2} = n\lambda^{2t-i-2}$ and $(1-p)^{(2t-i-2)n} = O(\exp(-\lambda (2t-i-2) ))$,  
    \begin{equation} 
        \bbE\bigl[N^2 \bigr] = \bigl(1+o(1)\bigr) \bbE\bigl[N\bigr]^2 + O\Bigl(\sum_{i\leq t-1} n \lambda^{2t-i-2} \exp\bigl(-\lambda (2t-i-2) ) \bigr)  \Bigr).\label{eq:YoungCobra}
    \end{equation}
    By \eqref{eq:WittyBall} and the discussion after \eqref{eq:WeepyNet}, we have $\bbE[N] =(1-o(1)) n c_{\lambda}\lambda^{t-1} \exp(-\lambda (t-1))$. 
    Here, $c_{\lambda}>0$ remains bounded away from zero by the assumption that $\lambda$ remains bounded away from one.
    It follows that $n \lambda^{2t-i-2} \exp(-\lambda (2t-i-2) ) )/ \bbE[N]^2 = O(n^{-1} \lambda^{-i} \exp(i\lambda ))$. 

    If $\lambda = o(\ln(n))$ then certainly $n^{-1} \lambda^{-i} \exp(i\lambda ) \to 0$ and we may conclude that all terms in the sum in \eqref{eq:YoungCobra} are of order $o(1) \bbE[N]^2$.
    Now suppose that $\lambda = \Theta(\ln(n))$. 
    By modifying the sequence $h$, we can further assume that equality holds in \eqref{eq:UnripeSong}. 
    Then,
    \begin{equation}
        \frac{n \lambda^{2t-i-2} \exp\bigl(-\lambda (2t-i-2) ) \bigr)}{\bbE[N]^2} = O\biggl(\frac{\exp(i\lambda)}{n\lambda^i } \biggr) = O\biggl( \exp\Bigl( \frac{i - (t-1)}{t-1} \ln(n) - i h(n) \Bigr)\biggr).\label{eq:UnripeTin}
    \end{equation}
    Note that $\exp[(i -(t-1))\ln(n)/(t-1)] = O(1)$ for $i\leq t-1$  and recall that $h(n) \to \infty$ to conclude that all terms in the sum in \eqref{eq:YoungCobra} are of order $o(1) \bbE[N]^2$. 
    This proves \eqref{eq:DryUnit}. 
\end{proof}

\begin{corollary}\label{cor: PendantTree}
    Fix a rooted tree $T$ on $t\geq 2$ vertices. 
    Assume that $\lambda>1$ remains bounded away from $1$ and that there exists a sequence $h(n)$ with $h(n)\to \infty$ such that \eqref{eq:UnripeSong} is satisfied. 
    Then, the giant component in $G(n,p)$ has a pendant copy of $T$ with probability tending to one: 
    \begin{equation}
        \lim_{n\to \infty} \bbP\bigl(\exists H: C_{\textnormal{giant}} \cong H \cdot T \bigr)  = 1. \label{eq:NiftyPop} 
   \end{equation}
\end{corollary}
\begin{proof}
    \Cref{lem: EN,lem: EN2} yield that $\bbE[N]\to \infty$ and $\bbE[N^2] = (1+o(1))\bbE[N]^2$. 
    Hence, Chebychev's inequality then yields that $N\to \infty$ in probability.  
    In particular, we then have $N>0$ with probability tending to one which yields \eqref{eq:NiftyPop}.  
\end{proof}

\begin{proof}[Proof of \texorpdfstring{\Cref{thm: GiantCospectral}}{Theorem}]
    Schwenk \cite{schwenk1973almost} gave rooted trees $T_1, T_2$ on $9$ vertices such that for every rooted graph $H$ the graphs $G_1 = H\cdot T_1$ and $G_2 = H \cdot  T_2$ are cospectral; see also \cite[p.65]{godsil1993algebraic}.
    Moreover, Godsil and McKay \cite[Section 9]{godsil1976some} found rooted trees $T_1', T_2'$ on $11$ vertices such that $H\cdot T_1'$ and $H \cdot  T_2'$ are $\bbR$-cospectral. 
    The desired result is hence immediate from \Cref{cor: PendantTree}. 
\end{proof}

Recall that \Cref{prop: ExponentiallyManyCospectral} assumed that $\lambda>1$ is fixed, and claimed that the giant component then has exponentially many cospectral mates. 
This again follows by swapping pendant trees:

\begin{proof}[Proof of \texorpdfstring{\Cref{prop: ExponentiallyManyCospectral}}{Proposition}]
    Let $T_1,T_2$ be the trees producing $\bbR$-cospectrality due to Godsil and McKay \cite{godsil1976some}.  
    Denote $\cV_{T_1}$ for the set of vertices in the $2$-core that have a pendant copy of $T_1$ attached.
    Then, we claim that there exists $c_1>0$ with 
    $
        \lim_{n\to \infty}\bbP(\#\cV_{T_1}\geq c_1n)=1. 
    $
    This could be deduced from a second moment calculation similar to \Cref{thm: GiantCospectral}, and also follows directly from a result by Ding, Lubetzky, and Peres  \cite[Theorem 1]{ding2014anatomy} that establishes that the giant is contiguous for fixed $\lambda>1$ to a model with independent Galton-Watson trees attached to the vertices in the $2$-core. 

    For any $S \subseteq \cV_{T_1}$ let $G_S$ denote the graph found from $C_{\textnormal{giant}}$ by replacing all copies of $T_1$ rooted in $S$ by a copy of $T_2$. 
    Then, $G_S$ is $\bbR$-cospectral to $C_{\textnormal{giant}}$.
    Further, if $\cV_{\textnormal{aut}}$ denotes the set of vertices in the $2$-core that admit a non-trivial orbit under automorphisms, then the graphs $\{G_{S} : S\subseteq \cV_{T_1}\setminus \cV_{\textnormal{aut}} \}$ are all pairwise non-isomorphic. 
    Consequently, 
    \begin{equation}
        \#\{G: G \textnormal{ is } \bbR\textnormal{-cospectral to }C_{\textnormal{giant}}  \} \geq \#\{G_{S} : S\subseteq \cV_{T_1}\setminus \cV_{\textnormal{aut}} \}  = 2^{\#\cV_{T_1} -\#\cV_{\textnormal{aut}}}.\label{eq:JumpyHat} 
    \end{equation}
    It is known due to Verbitsky and Zhukovskii \cite[Theorem 1.2]{verbitsky2024canonical} that $\#\cV_{\textnormal{aut}}$ is bounded in probability, meaning that $\lim_{C\to \infty} \limsup_{n\to \infty} \bbP(\#\cV_{\textnormal{aut}} >C) = 0$. 
    Consequently, it holds for any $\varepsilon>0$ that $\lim_{n\to \infty}\bbP(\#\cV_{T_1} -\#\cV_{\textnormal{aut}} \geq (c_1 -\varepsilon)n) = 1$. Combine this with \eqref{eq:JumpyHat} to conclude. 
\end{proof}

\section{A switching method in a 2-core needs at least two cycles -- Proof of \texorpdfstring{\Cref{thm: NoSwitching}}{Theorem}}\label{sec: SwitchingFails}

Recall from \Cref{def: SwitchingOfSizeM} that a switching method with size $\leq m$ involves a conjugation of the adjacency matrix with $Q \oplus I_{n-m}$ for some $m\times m$ orthogonal matrix $Q$. 
It is here allowed that $Q$ may have integral rows, which is the reason for the inequality as we could potentially achieve the same operation with a smaller matrix.
More precisely, due to orthogonality, any integral row of $Q$ has a single entry $\pm 1$ and only zeros in the remaining entries. 

If all rows of $Q$ are integral, then it is simply a signed permutation matrix and the signs may be taken positive. 
In that case, the switching will result in a graph that is isomorphic to the original graph. 
Let us hence suppose that $Q$ has non-integral rows and denote $\tilde{Q}$ the largest principal submatrix of $Q$ that has no integral rows. 
Then, we refer to the operation as a \emph{$\tilde{Q}$-switching}. 
The vertices of $G$ corresponding to this submatrix are called the \emph{switching set}. 
Thus, in the Godsil--McKay switching of \Cref{ex: GosilMcKay} the switching set was $X$.

From here on, our main focus will be the non-integral submatrix $\tilde{Q}$. 
We hence drop the tilde from here on. 
Our main technical result is \Cref{thm:noswitching2core}:

\begin{theorem}\label{thm:noswitching2core}
Let $G$ be a connected graph with minimum degree $\geq 2$ and consider an orthogonal matrix $Q\in \bbR^{k\times k}$ with no integral entries. 
Suppose that $G$ has a non-isomorphic cospectral mate through a $Q$-switching. 
Then, $G$ has a connected subgraph with at most $2k+1$ vertices and strictly more edges than vertices.  
\end{theorem}

Note that a connected graph with more edges than vertices, must have at least two cycles. 
Hence, we could also view \Cref{thm:noswitching2core} as stating that $G$ must have a `small' connected subgraph with at least two cycles.

In the proof of \Cref{thm:noswitching2core} the required subgraph is found in the neighborhood graph of the switching set $X$; see \Cref{def: NeighbourhoodGraph}. 
We show in \Cref{sec: Cycles} that the switching method will result in an isomorphic graph if the neighborhood graph is a disjoint union of cycles, which is used in the proof of \Cref{thm:noswitching2core} to ensure that a non-isomorphic switching requires at least one vertex with degree $\geq 3$ in the neighborhood graph. 
The latter vertex is used as the starting point to construct the connected subgraph from the statement of \Cref{thm:noswitching2core}; see \Cref{sec: ProofsSwitching}. 

Given \Cref{thm:noswitching2core}, the proof of \Cref{thm: NoSwitching} follows readily since small subgraphs with more edges than vertices are improbable when $\lambda$ is small. 
We refer to \Cref{sec: ProofsSwitching} for the details.  

\subsection{Switching methods induce cospectrality on subgraphs}\label{sec: CospectralityOnSubgraphs}
Recall that a switching method only modifies a small subgraph and the edges connected to that subgraph.
This will be made precise using the following \Cref{def: NeighbourhoodGraph}, visualized in \Cref{fig: Sx}.

\begin{definition}\label{def: NeighbourhoodGraph}
    Let $G = (V,E)$ be a graph and $X$ be a subset of the vertices. 
    The \emph{neighborhood graph} of $G$ for $X$ is the graph $G_{X}$ that has vertex set $X \cup N(X)$ with $N(X) \de \cup_{x\in X}\{v\in V: v\sim x \}$ and edge set $\{e \in E : e \cap X \geq 1\}$.
\end{definition}

\begin{figure}[hbt]
    \centering
    \includegraphics[width=0.5\linewidth]{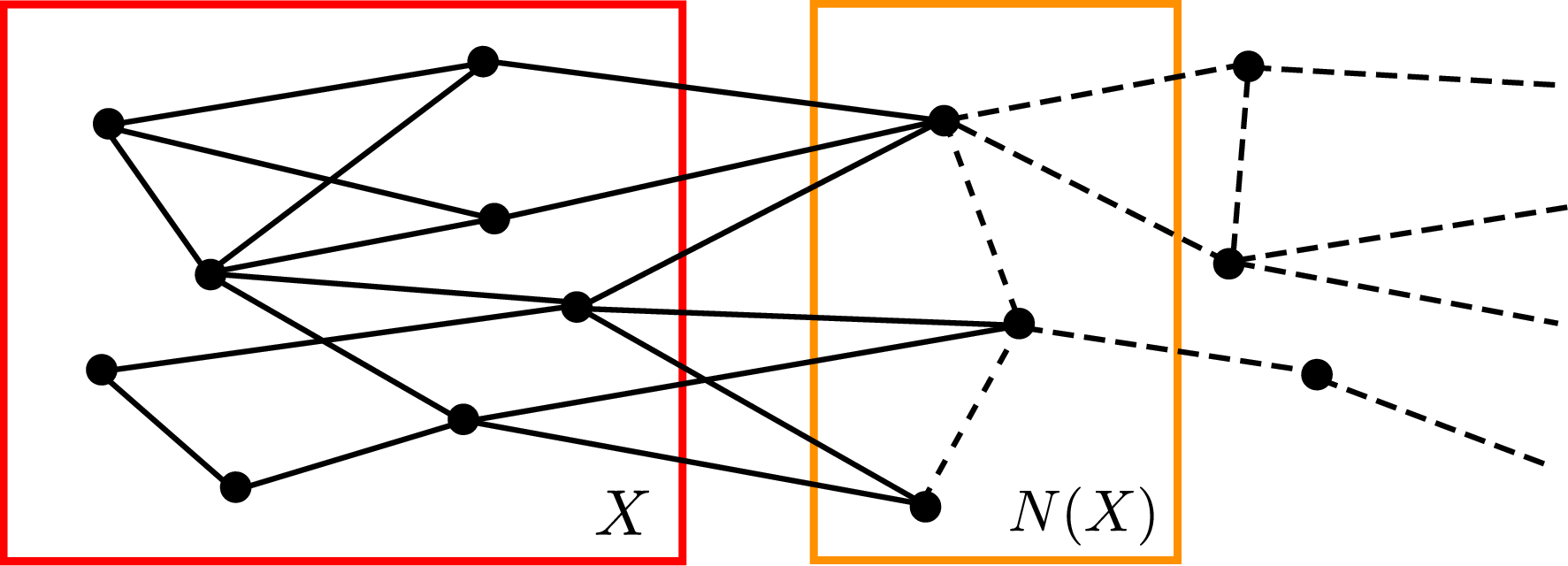}
    \caption{The neighborhood graph of $X$ includes the edges with at least one endpoint in $X$.}
    \label{fig: Sx}
\end{figure}

When we consider a graph $H$ obtained from $G$ by a switching operation in the subsequent arguments, then we always also equip $H$ with the vertex ordering associated to the switching method that is referred to in \Cref{def: SwitchingOfSizeM}. 
We regard the vertex to be the same.
In particular, given a set of vertices $S \subseteq V$ the induced subgraphs on this set, denoted $G[S]$ and $H[S]$, are defined for both $G$ and $H$.

If a graph admits a $Q$-switching, then every subgraph that contains the neighborhood graph of the switching set also admits a $Q$-switching: 

\begin{lemma}\label{lem:subgraphcospectral}
Consider graphs $G = (V,E)$ and $H$ with $H$ obtained from $G$ by a $Q$-switching with switching set $X$. 
If $Y \subseteq V\setminus X$ and $F$ is a set of edges of $G[X\cup Y]$ with $F \cap E(G_{X}) = \emptyset$, then the $Q$-switching can be applied to $G[X \cup Y]-F$ with the cospectral graph $H[X \cup Y] - F$ as a result.
\end{lemma}
\begin{proof}
Up to reordering the vertices the $Q$-switching is represented by a block diagonal matrix with diagonal blocks $Q$, $-I$, and $I$. 
Consider the block matrix decomposition of the adjacency matrix corresponding to this ordering of the vertices:
\begin{equation}
    A(G) \ed \begin{pmatrix}  B & U & W\\ U^T & C_{11} & C_{12}\\ W^T  & C_{21} & C_{22}\end{pmatrix}.\label{eq:HugeVolt}
\end{equation}
Then, the $Q$-switching applied to $G$ can be represented as the conjugation,
\begin{equation}
    \begin{pmatrix} Q & O & O\\ O & -I & O \\ O & O & I  \end{pmatrix}^T\begin{pmatrix} B & U & W\\ U^T & C_{11} & C_{12}\\ W^T  & C_{21} & C_{22}\end{pmatrix}\begin{pmatrix} Q & O & O\\ O & -I & O \\ O & O & I  \end{pmatrix} =  \begin{pmatrix} Q^TBQ & -Q^TU& Q^TW\\ -U^TQ & C_{11} & -C_{12}\\ W^TQ  & -C_{21} & C_{22}  \end{pmatrix} = A(H).\label{eq:TallHotel}
\end{equation}

The adjacency matrix of $G[X\cup Y]$ is obtained from \eqref{eq:HugeVolt} by removing all rows and columns that are not associated with vertices from $X\cup Y$.
Thus, for example, in the first block of rows in \eqref{eq:HugeVolt} we retain the first block $B$ completely as this corresponds to the switching set $X$, and retain only those columns of $U$ and $W$ that are associated to $Y$.   
Moreover, removing the edges in $F$ only changes entries in the $C$-submatrices from a one to a zero.
Thus, 
\begin{equation}
    A\Bigl(G[X\cup Y] - F\Bigr) = \begin{pmatrix}  
    B & \tilde{U}_{Y} & \tilde{W}_Y\\ 
    \tilde{U}^T_Y & \tilde{C}_{11} & \tilde{C}_{1,2}\\ 
    \tilde{W}^T_Y  & \tilde{C}_{2,1} & \tilde{C}_{22}\end{pmatrix}, \label{eq:GrumpyEgg}
\end{equation}
with $\tilde{U}_Y$ and $\tilde{W}_Y$ the matrices found from $U$ and $W$ by removing the appropriate columns, and $\tilde{C}_{i,j}$ found from $C_{i,j}$ by removing rows and columns and potentially zeroing out certain entries.

Now consider the orthogonal matrix found from $\operatorname{diag}(Q,-I,I)$ by only retaining the rows and columns associated to $X \cup Y$. 
This matrix has the form $\operatorname{diag}(Q, -I_a, I_b)$ with $a+b = \#Y$.
It then follows from \eqref{eq:TallHotel} that the adjacency matrix of $H[X \cup Y] - F$ is precisely the matrix found by conjugating \eqref{eq:GrumpyEgg} with $\operatorname{diag}(Q, -I_a, I_b)$. 
Thus, $G[X\cup Y] - F$ and $H[X\cup Y]-F$ are indeed related by a $Q$-switching. 
\end{proof}

\subsection{Disjoint union of cyles results in isomorphism}\label{sec: Cycles}
The difficult case in the proof of Theorem~\ref{thm:noswitching2core} will occur when $G_X$ is a disjoint union of cycles.
This challenge cannot be ignored since multiple small cycles far apart from each other can appear with nonzero probability in the $2$-core.
Our goal here is to show that in this case the switching always results in an isomorphic graph, which we accomplish in \Cref{lem:switchingcyclesiso}. 

\subsubsection{Generating cospectrality for graphs with maximum degree 2}\label{sec: GenCospec}

We will analyze the subgraphs of $G_X$, which are graphs with maximum degree at most two when $G_X$ is a disjoint union of cycles. 
\Cref{thm:pathscycles} below shows that all cospectral graphs among these are generated by a small set of cospectral pairs in the following sense:

\begin{definition}\label{def: Generated}
 Let $\mathcal{F} = \{(F_i, F_i') : i \in I\}$ be a set of cospectral pairs. The cospectrality between two graphs $G$ and $H$ is \emph{generated} by $\mathcal{F}$ if there is a multisubset $J$ of $I$ such that $G \oplus_{j \in J} F_j \cong H \oplus_{j \in J} F'_j$. 
 Here, $\oplus$ indicates disjoint union.
\end{definition}

The specific problem of generating all cospectral graphs that are subgraphs of unions of cycles has been analyzed by Cvetkovi\'c, Simi\'c and Stani\'c \cite{cvetkovic2010pathsandcycles}. 
However, the stated result there only identified the `minimal non-DS graphs'. 
A priori this does not exclude the existence of a pair of cospectral graphs, both unions of paths and cycles, that is not included in the list in \cite[Theorem 3.1]{cvetkovic2010pathsandcycles} because neither graph is minimal non-DS.
A claim without minimality however follows readily from an older result of \cite{Cvetkovic1975} concerning the generation of all graphs with spectral radius at most two: 

\begin{theorem}[Cvetkovi\'c and Gutman \cite{Cvetkovic1975}]\label{thm:pathscycles}
    Let $\Gamma,\Gamma'$ be a pair of cospectral graphs with maximum degree bounded by $2$. Then the cospectrality is generated by $\{(C_{2n} + 2P_1, C_{4} + 2P_{n-1}) : n \geq 3\} \cup \{(C_{4} + 2P_{n-1}, C_{2n} + 2P_1) : n \geq 3\} $
\end{theorem}
\begin{proof}
    The spectral radius is bounded by the maximum degree, so the graphs we consider have spectral radius at most $2$.  
    Equations (2a)-(2i) in \cite{Cvetkovic1975} gave a set of pairs of graphs $\mathcal{F}$ that generate all cospectralities among graphs with spectral radius at most $2$. Hence there are some $(F_j,F_j')_{j \in J} \in \mathcal{F}$ such that $\Gamma \oplus_{j \in J} F_j \cong \Gamma' \oplus_{j \in J} F_j'$. Each component that is not a cycle or path only appears in one equation among (2a)-(2i). 
    Hence, since $\Gamma,\Gamma'$ have maximum degree bounded by $2$, it holds for any pair $(F_j, F_j')$ with $j \in J$ including such a component that there must be a $k \in J$ such that $(F_k, F_k') = (F_j',F_j)$. We could remove $j,k$ from $J$ and still have a valid isomorphism.     
    Only Equation (2c) in \cite{Cvetkovic1975}, which involves $F_j \de C_{2n}+2P_1$ and $F_j'\de  C_4+2P_{n-1}$, does not involve any components with maximum degree $\geq 3$.  
    We conclude that this pair must generate all pairs of cospectral graphs with maximum degree bounded by $2$.
\end{proof}

\pagebreak[3]
It follows that cospectral mates with maximum degree $2$ differ by at least three components:
\begin{lemma}\label{lem: CospectralMaximumDegree2}
Let $\Gamma, \Gamma'$ be non-isomorphic cospectral graphs both with maximum degree bounded by $2$. 
Then, there exist $n,k\geq 2$ with $n\neq k$ such that $\Gamma$ has at least one more $C_{2n}$ and two components $P_{k-1}$ more than $\Gamma'$ and $\Gamma'$ has a $C_{2k}$ and two $2P_{n-1}$ components more than $\Gamma$.
\end{lemma}

\begin{proof}
    Let $\Gamma = \sum_{i = 1}^s a_iC_i + \sum_{j = 1}^s b_jP_j$ and $\Gamma' = \sum_{i = 1}^s a_i'C_i + \sum_{j = 1}^s b_j'P_j$ for some large enough $s$. 
    We claim that it then follows that
    \begin{equation}\label{eq:ndspathscycles}
        2a_{2i}+b_{i-1} = 2a_{2i}'+b_{i-1}' \text{ for all } 1 \leq i \leq s +1.  
    \end{equation} 
    Indeed, \eqref{eq:ndspathscycles} holds true if $\Gamma = C_{2n} + 2P_1$ and $\Gamma' = C_{4} + 2P_{n-1}$, and the general case then follows from \Cref{thm:pathscycles} by considering unions; recall \Cref{def: Generated}. 
    
     The graphs $\Gamma$ and $\Gamma'$ are cospectral with the same number of cycle components but not isomorphic. 
     \Cref{thm:pathscycles} implies that they cannot differ in the odd cycles and \eqref{eq:ndspathscycles} implies that they cannot only differ in the paths, so there is an $n \geq 2$ such that $a_{2n} > a_{2n}'$ and a $k \neq n$ such that $a_{2k} < a_{2k}'$. From \eqref{eq:ndspathscycles} for $i = n,k$ we get that there is $n,k$ such that the lemma holds.
\end{proof}

\begin{lemma}\label{lem:addingtounionofpaths}
    Let $G$ and $H$ be cospectral graphs with maximum degree $2$.
    Assume that there exists vertices $x\in G$ and $y\in H$ such that the subgraphs $G\setminus x$ and $H \setminus y$ are isomorphic disjoint unions of paths. 
    Then, $G$ and $H$ are isomorphic.
\end{lemma}

\begin{proof}
    The graphs $G$ and $H$ can be formed from a union of paths by adding a vertex with degree at most $2$, so they can have at most one cycle.
    Assume that $G$ and $H$ are not isomorphic. 
    Then, \Cref{lem: CospectralMaximumDegree2} implies that there exist $n, k \geq 2$ with $n \neq k$ and 
    \begin{equation}
        G = C_{2n}  + F \qquad \textnormal{ and }\qquad H = C_{2k} +   F',\label{eq:LazyPop}
    \end{equation}
    for unions of paths $F$ and $F'$ with the property that $F$ has at least two components $P_{k-1}$ more than $F'$. 
    However, using the assumption that the vertex-deleted graphs are isomorphic unions of paths in \eqref{eq:LazyPop} implies that $P_{2n-1} +  F \cong P_{2k-1}  +F'$. 
    Thus, $F$ and $F'$ can differ by at most one component, a contradiction.  
\end{proof}

\begin{remark}
     \Cref{lem:addingtounionofpaths} can also be rephrased in terms of so-called overgraphs. As such it is similar to \cite[Theorem~16]{farrugia2019overgraphs}, but with the condition of maximum degree $2$ instead of controllability.
\end{remark}

\subsubsection{No non-isomorphic switching if \texorpdfstring{$G_X$}{GX} is a union of cycles}

\begin{lemma}\label{lem:cyclesDS}
    Let $X$ be the switching set of a $Q$-switching in a graph $G$.  
    Suppose that the neighborhood graph $G_{X}$ is a union of cycles and that $G_{X}[X]$ is a union of paths. 
    Then, with $H$ the graph obtained by switching, $G_{X}[X \cup Y]$ is isomorphic to $H_{X}[X \cup Y]$ for every $Y \subseteq N(X)$.
\end{lemma}

\begin{proof}
     A union of cycles is determined by its spectrum (\eg by \cite{Cvetkovic1975}). 
     Hence,  since \Cref{lem:subgraphcospectral} implies that $H_X$ is cospectral to $G_X$, it holds that $H_X$ is also a union of cycles and has maximum degree $2$. 
     Further, again by \Cref{lem:subgraphcospectral}, $G_{X}[X \cup Y]$ and $H_{X}[X \cup Y]$ are cospectral for all $Y \subseteq N(X)$. We show by induction on $\#Y$ that $G_{X}[X \cup Y]$ and $H_{X}[X\cup Y]$ are in fact isomorphic.
     
     A union of paths has no cycle components, so among graphs with maximum degree at most two it is determined by its spectrum by Lemma~\ref{lem: CospectralMaximumDegree2}.
     Hence, $G_{X}[X]$ and $H_{X}[X]$ are isomorphic. 
     
     Now consider $k\geq 1$ and suppose that $G_{X}[X \cup Y] \cong H_{X}[X\cup Y]$ for all $Y \subseteq N(X)$ with $\#Y < k$. 
     Pick some $Y \subseteq N(X)$ with $\#Y = k$. 
     Let $y \in Y$ be any element. 
     Then, $G_X[X \cup Y\setminus\{y\}] \cong H_X[X\cup Y\setminus\{y\}]$ by the induction hypothesis. 
     Let $Z$ and $Z'$ be the sets of vertices in $G_X[X \cup Y\setminus\{y\}]$ and $H_X[X \cup Y\setminus\{y\}]$ respectively that are not a part of a cycle component. 
     Then, $G_X[Z]$ and $H_X[Z']$ are obtained by deleting the same cycle components and hence isomorphic unions of paths. 
     Further, $G_X[Z \cup y]$ and $H_X[Z'\cup y]$ are cospectral, because they are obtained from the cospectral graphs $G_{X}[X \cup Y]$ and $H_{X}[X\cup Y]$ by deleting an isomorphic set of cycle components. 
     Lemma~\ref{lem:addingtounionofpaths} now yields that $G_X[Z \cup y] \cong H_X[Z'\cup y]$. 
     As $y$ is not adjacent to any of the cycle components in $G_X[X \cup Y\setminus\{y\}]$ or $H_X[X \cup Y\setminus\{y\}]$, adding the cycle components back now yields that $G_{X}[X \cup Y]\cong H_{X}[X\cup Y]$. This completes the induction.
\end{proof}

\pagebreak[3]
\begin{lemma}\label{lem:switchingcyclesiso}
    Let $X$ be the switching set of $Q$-switching in a connected graph $G$.
    If $G_{X}$ is a disjoint union of cycles, then $Q$-switching creates an isomorphic graph to $G$.
\end{lemma}
\begin{proof}
     Let $H$ be the cospectral mate of $G$ created by $Q$-switching on $X$.
     If $G_{X}[X]$ contains a cycle, then this is one of the components of $G_{X}$, say $C$, by the assumption that $G_{X}$ is a union of cycles. 
     However, $C$ is then also a component of $G$ since the neighborhood graph $G_X$ contains all edges leading to vertices in $X$. 
     The graph $G$ is connected, so $C = G$. 
     Then $G$ is isomorphic to $H$, because cycles are determined by spectrum.
     
     Now suppose that that $G_{X}[X]$ is a union of paths. 
     Then, for every $Y \subseteq N(X)$ the graph $G_{X}[X \cup Y]$ is a union of paths and cycles and \Cref{lem:switchingcyclesiso} yields that $G_{X}[X \cup Y] \cong H_{X}[X \cup Y]$. 
     In particular, $G_{X}[X] \cong H_{X}[X]$. 
     Let $P(X)$ denote the set of components of $G_{X}[X]$, which is also the set of components in $H_X[X]$. 
     We claim that it suffices to construct permutations $\phi : P(X) \to P(X)$ with the following properties
    \begin{enumerate}
        \item every path $A\in P(X)$ gets mapped to a path of the same length,
        \item if $y \in N(X)$ has $i$ edges to $A$ in $G$, then $y$ has $i$ edges to $\phi(A)$ in $H$ for $0 \le i \le 2$. 
    \end{enumerate}
    Specifically, we claim that such a $\phi$ induces an isomorphism between $G$ and $H$.

    The first property implies that there are isomorphisms between $A$ and $\phi(A)$ for every path $A\in P(X)$. As $G_X$ and $H_X$ are unions of cycles, the only edges going out of $A$ come from the endpoints of the path. 
    Hence, for any path $A$ the set of neighbors in $G$ has the form $\{y,z \}$, possibly with $y = z$.
    The second property implies that at least one of the isomorphisms\footnote{A path of nontrivial length has automorphism group of order $2$.
    Thus, there are typically two isomorphisms between $A$ and $\phi(A)$, except if $A$ is a point in which case there is only one. } between $A$ and $\phi(A)$, say $\psi_A$, sends the neighbors of $y$ and $z$ in $A \subseteq G$ to their neighbors in $\phi(A)\subseteq H$. 
    The map $\psi : V(G) \to V(H)$ which acts on a path $A \in P(X)$ as $\psi_A$ and fixes all other vertices is then indeed an isomorphism between $G$ and $H$. 
    Indeed, it acts as the identity on $V\setminus X$, the permutation $\phi$ ensures that it acts by an isomorphism on $X$, and the choices of $\psi_A$ ensure that any edges between $X$ and $V\setminus X$ also get mapped to edges. 
    
    We next construct the permutation $\phi$. 
    Let $A \in P(X)$ be a path of length $a$. 
    In $G_{X}$ there are no paths, so $N_G(A) \neq \emptyset$. If $N_G(A) = \{y\}$. 
    Then $G_{X}[X \cup y] \cong H_{X}[X \cup y]$ has one cycle of length $a+1$, so $y$ must also be twice adjacent to a path of length $a$ in $H_{X}$. Define $\phi(A)$ to be this path.

    Now suppose that $N_G(A) = \{y,z\}$. The component, $T$, containing $A$ in $G_X[X\cup \{y,z\}]$ must be a path or cycle, because $G_X[X\cup \{y,z\}]$ is a subgraph of $G_X$, which is a union of cycles. Moreover, by \Cref{def: NeighbourhoodGraph} $y$ and $z$ are not adjacent and must have degree $2$ in $G_X[X\cup \{y,z\}]$. Thus $T$ can be decomposed as $B$\nobreakdash-$y$\nobreakdash-$A$\nobreakdash-$z$\nobreakdash-$C$ with $B = C$ if and only if it is a cycle; see \Cref{fig:siwtchingkeepsmiddlepath}. Note that $G_X[X\cup \{y,z\}]$ has one more component isomorphic to $T$ than $G_X[X]$, $G_X[X\cup y]$ or $G_X[X\cup z]$. 
    The isomorphisms $G_{X}[X\cup Y] \cong H_{X}[X\cup Y]$ imply that the same holds for $H$, so $y$ and $z$ are in the same component of $H_X[X\cup \{y,z\}]$.
    
    Additionally, adding $y$ to $G_X[X]$ removes paths of length $a\de \# A$ and $b \de \#B$, and adds one of length $a+b+1$. 
    It then follows from $G_X[X]\cong H_X[X]$ and $G_X[X \cup y]\cong H_X[X\cup y]$ that $y$ must be adjacent to paths $A_y, B_y$ of length $a$ and $b$ in $H_{X}[X \cup y]$. 
    Similarly, $z$ is adjacent to paths $A_z,C_z$ of length $a$ and $c = \# C$ in $H_X[X \cup z]$. 
    Moreover, since $y$ and $z$ are in the same component of $H_X[X\cup \{y,z \}]$, at least one of these paths must connect $y$ and $z$. 
    We claim moreover that there must be a connecting path of length $a$. 
    Indeed, the only case in which this would not occur is if $A_y \neq A_z$ and $C_z = B_y$ with $b = c \neq a$. 
    Then, the component $T$ of $H_X[X\cup \{y,z \}]$ containing $y,z$ is a path of length $2a + b +2$, which contradicts that it has length $a+b+c+2$ by the decomposition $B$\nobreakdash-$y$\nobreakdash-$A$\nobreakdash-$z$\nobreakdash-$C$ in $G_X[X\cup \{y,z \}]$ with $B\neq C$ when $T$ is a path.  
    
    Hence there is at least one path of length $a$ in $H_X$ that is adjacent to $y$ and $z$. If there is exactly one such path define $\phi(A)$ to be this path. 
    In case there are two such paths, which may happen if $B=C$ and $b = a$, then $y$ and $z$ are part of a cycle  $z$\nobreakdash-$P_a$\nobreakdash-$y$\nobreakdash-$P_a$\nobreakdash-$z$ in $H_{X}$ and hence in $G_{X}$. 
    Let $A$ and $\tilde{A}$ be these $P_a$ paths in $G_{X}$. Map $A$ and $\tilde{A}$ to the two $P_a$ paths in $H_{X}$ adjacent to $y,z$ (either order works). The map we have defined here is injective    and must therefore be a permutation. It has the required properties and hence extends to an isomorphism between $G$ and $H$.
\end{proof}
    \begin{figure}[htb]
        \centering
        \includegraphics[width=0.73\linewidth]{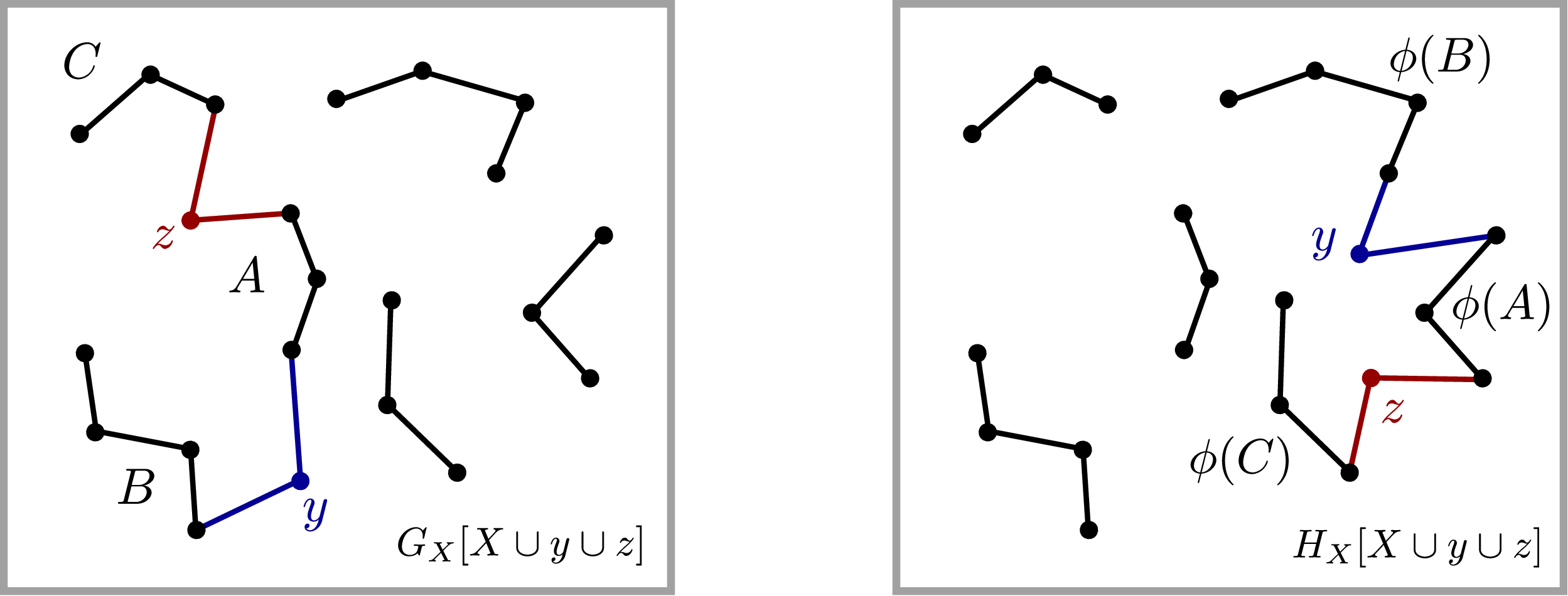}
        \caption{The proof of \Cref{lem:switchingcyclesiso} constructs an isomorphism $G\cong H$ based on a permutation $\phi$ of the paths in $G_X[X]$ and $H_X[X]$. 
        This permutation has the property that for every $y,z \in N(X)$ adjacent to a path $A \subseteq G_X[X]$ in $G$, it holds that the path $\phi(A) \subseteq H_X[X]$ is adjacent to $y,z$ in $H$. }
        \label{fig:siwtchingkeepsmiddlepath}
    \end{figure}

\subsection{Proofs of Theorems \texorpdfstring{\ref{thm:noswitching2core}}{?} and \texorpdfstring{\ref{thm: NoSwitching}}{?}}\label{sec: ProofsSwitching}

\begin{proof}[Proof of Theorem~\ref{thm:noswitching2core}]

 Consider some $k\times k$ orthogonal matrix $Q$ with no integral rows such that the $Q$-switching creates a cospectral mate.
 Let $X$ be the switching set.
 Then $Q^Tv \notin \{0,1\}^k$ for every vector $v \in \{0,1\}^k$ with exactly one entry equal to one, because $Q$ has no integral rows. 
 Thus, any vertex adjacent to a switching set $X$ has at least two neighbours in $X$.
 The assumption that $G$ has minimum degree $\geq 2$ hence implies that the neighborhood graph $G_X$ has again has minimum degree $\geq 2$ since vertices in $X$ keep the same degree, while all vertices in $N(X)$ have at least two neighbours in $X$. 

 Graphs for which all vertices have degree $2$ are disjoint unions of cycles. 
 Since the switching creates a non-isomorphic graph, \Cref{lem:switchingcyclesiso} implies that there must be a vertex $v_0$ in $G_{X}$ with $\deg(v_0) > 2$. 
 Now take a cycle $C_0$ of $G_{X}$ containing $v_0$.
 By definition of $G_{X}$, no pair of vertices outside $X$ are adjacent. 
 Hence, 
 \begin{equation}
    \# (C_0 \cap X) \geq \#(C_0 \setminus X).
 \end{equation}
 We next extend $C_0$ to obtain a graph with at most $2\#X+1$ vertices and more edges than vertices.

 Given that $v_0$ has degree $>2$ in $G_X$, there must be a vertex $v_1$ outside of $C_0$ adjacent to $v_0$. 
 Consider a path $P = v_1v_2,\ldots, v_i$ in $G_X\setminus C_0$ that starts from $v_1$ and can not be extended to a longer path. 
 That the path can not be extended means that all neigbours of $v_i$ in $G_X$ are in $C_0 \cup \{v_1,\ldots,v_{i-1} \}$.  
 Given that $G_X$ has minimal degree $\geq 2$, it must then hold that $v_i$ has at least two neighbors in $C_0 \cup \{v_1,\ldots,v_{i-1} \}$. 
 (See \Cref{fig:placeholder2}.)
 Either way the subgraph of $G_X$ induced by $C_0 \cup \{ v_1,\ldots, v_i\}$ is connected and has at least $\#C_0 + i +1 $ edges since the path $P$ has $i-1$ edges, and we gain two edges from the connections with $v_1$ and $v_i$.

 Further, the path satisfies that $\#(P \cap X)+1 \geq \#(P\setminus X)$ since every vertex outside $X$ must be followed by one in $X$, except for possibly the final one.  
 Hence, the number of vertices in the constructed subgraph satisfies that  
 \begin{equation}
    \#C_0 \cup P = \# (C_0 \cup P) \cap X + \#(C_0 \cup P_1)\setminus X \le 2\#(C_0 \cup P_1) \cap X +1 \le 2\#X+1. 
 \end{equation}
 Thus, $G$ has a connected subgraph on $\leq 2\#X +1$ vertices with strictly more edges than vertices, namely the graph induced by $C_0 \cup P$.
 This concludes the proof.
\end{proof}
 \begin{figure}[htb]
    \centering
    \includegraphics[width=0.36\linewidth]{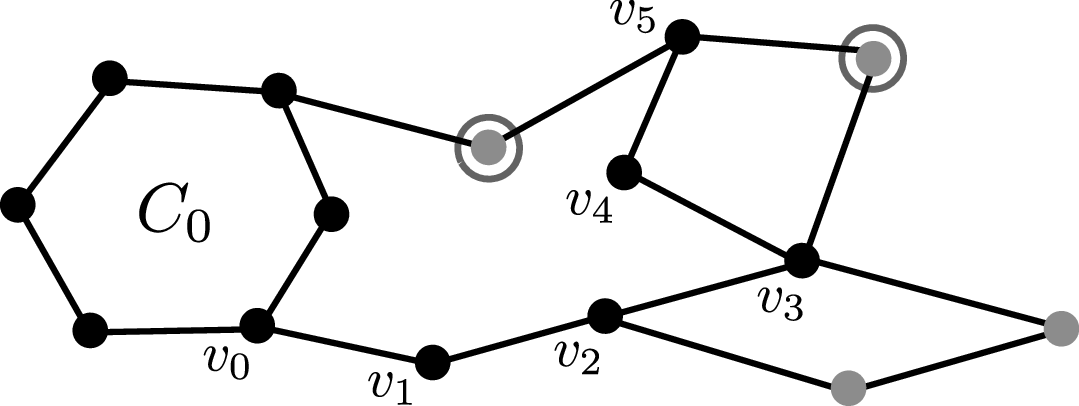}
    \caption{
    The proof of \Cref{thm:noswitching2core} constructs the required subgraph by considering a cycle $C_0 \subseteq G_X$ with a vertex $v_0$ of degree $\geq 3$ and creates a second cycle by considering maximal path in $G_X\setminus C_0$ starting from a neighbour $v_1$ of $v_0$.
    In the depicted example, such a path $v_1v_2\ldots v_6$ can be found by taking $v_6$ to be one of the circled vertices. }
    \label{fig:placeholder2}
\end{figure}

\begin{remark}
    The condition that the graph has minimum degree $2$ is essential in \Cref{thm:noswitching2core}. 
    Even with only one vertex of degree $1$ it is possible for a switching method to create non-isomorphic graphs with only one cycle in the neighborhood graph $G_X$; see \Cref{fig:GM} for example. 
    \begin{figure}[h!]
\centering
\begin{subfigure}{.49\textwidth}
  \centering
  \includegraphics[width=0.6\linewidth]{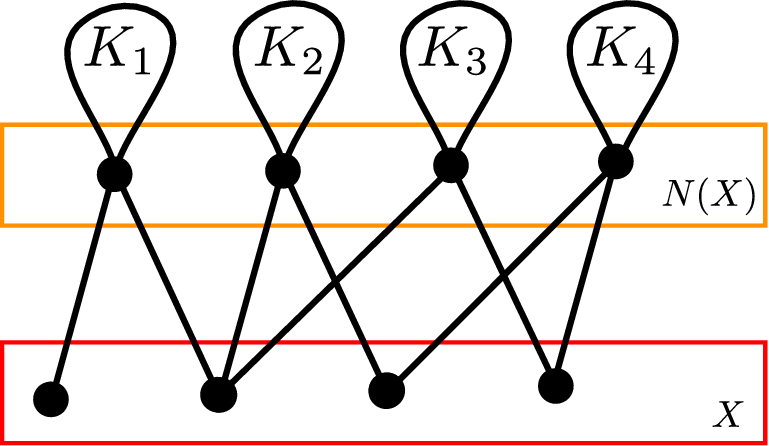}
\end{subfigure}% 
\begin{subfigure}{.49\textwidth}
  \centering
  \includegraphics[width=0.6\linewidth]{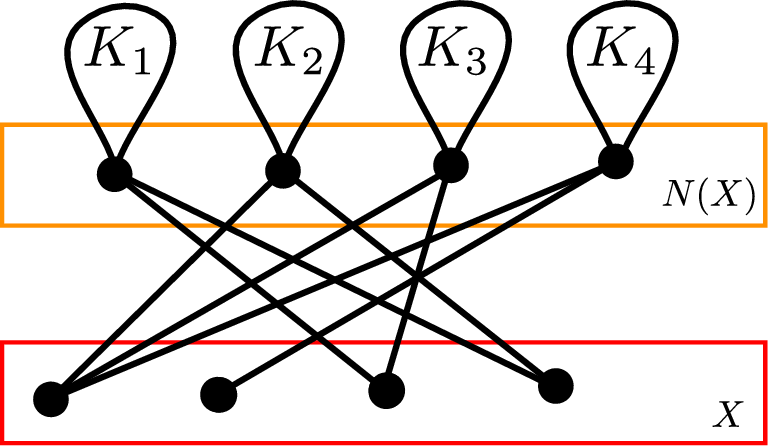}
\end{subfigure}
\caption{
(Left) In this example, there is a single cycle in the neighborhood graph $G_X$ and a single node of degree one. 
Non-isomorphic pending graphs $K_1,\ldots,K_4$ are attached to $N(X)$. 
(Right) Godsil--McKay switching (recall \Cref{ex: GosilMcKay}) yields a non-isomorphic graph, as may be verified by considering the degrees of the neighbors of the unique vertex with a pending $K_4$. }
\label{fig:GM}
\end{figure}
\end{remark}

\begin{proof}[Proof of \Cref{thm: NoSwitching}]

    Recall from the start of \Cref{sec: SwitchingFails} that any switching of size $\leq m$ corresponds to a $Q$-switching for some $k\times k$ orthogonal matrix $Q$ with $k\leq m$.  
    Hence, since the $2$-core of a connected graph is always again a connected graph and has minimum degree $\geq 2$, it follows from  \Cref{thm:noswitching2core} that   
    \begin{align}
        \bbP&\bigl(\operatorname{Core}_2(C_{\textnormal{giant}})\textnormal{ admits a switching of size } \leq m \textnormal{ that yields a non-isomorphic graph}  \bigr)\label{eq:PinkGoose}\\ 
        &\qquad \qquad \qquad \leq \sum_{k\leq 2m+1} \bbP(\operatorname{Core}_2(C_{\textnormal{giant}}) \textnormal{ has a subgraph with }k\textnormal{ vertices and } \geq k+1 \textnormal{ edges}).\nonumber
    \end{align}
    Here, for any graph $H$ on $k$ vertices with $e$ edges, it holds with $G$ the original $G(n,p)$ graph that  
    \begin{equation}
        \bbP\Bigl(H \subseteq G \Bigr) \leq \sum_{v_1,\ldots,v_k = 1}^n \bbP\Bigl(G[v_1,\ldots,v_k] \textnormal{ has at least } e \textnormal{ edges} \Bigr) = O\Bigl(n^k p^e\Bigr). 
    \end{equation}
    Note that $n^{k}p^e = \lambda^e / n^{e-k} $ tends to zero if $\lambda = o(n^{\frac{e-k}{e}})$. 
    In particular, for any fixed $k$, it follows that with high probability the $G(n,p)$ graph (and hence also the $2$-core of its giant component) has no subgraphs with $\leq k$ nodes and $\geq k+1$ edges if $\lambda = o(n^{1/(k+1)})$.
    The assumption that $\lambda = o(n^{1/(2m+2)})$ now yields that \eqref{eq:PinkGoose} tends to zero as $n\to \infty$, concluding the proof . 
    
\end{proof}

\section{Numerical evidence}\label{sec: NumericalEvidence}  

We finally consider numerical evidence for a variant of Conjecture~\ref{conj: 2Core} with adjacency cospectrality replaced by $\bbR$-cospectrality. 
This setting allows the application of algorithms by Wang and Wang that can certify that a given graph is determined by its $\bbR$-spectrum, or search for an $\bbR$-cospectral mate if it is not  \cite{wang2025haemers}. 
Those algorithms are applicable to graphs that are controllable or almost controllable, meaning that the matrix $W \in \bbZ^{n\times n}$ with columns $e, A e,\ldots,A^{n-1}e$ must satisfy $\rank(W)\geq n-1$ where $A$ is the adjacency matrix and $e = (1,\ldots,1)^{\T}$.  

\subsection{Main findings}

We generated random graphs for varying average degree $\lambda > 1$ and number of nodes $n$ and applied the algorithms from \cite{wang2025haemers}, with some additional optimizations discussed below in \Cref{sec: AlgoOptimization}.  
The results are presented in \Cref{tab:2coreexperiments}. 
Our conclusions are as follows: 
\begin{description}[leftmargin = 1.5em]
    \item[On the frequency of $\bbR$-cospectrality]  From the total $100\times 22$ sampled graphs, the giant's $2$-core was only certified to be $\bbR$-cospectral for a total of $6$ samples while many samples were proven to be determined by the $\bbR$-spectrum. 
    The main caveat is that there is a nontrivial fraction of graphs where the algorithm is inconclusive for small $\lambda$. 
    For moderate-sized average degree we see that the algorithm certifies that many graphs are determined by their $\bbR$-spectrum; this was the case for all but one samples for $\lambda = 5$.   
    \item[Non-controllability]
    The primary culprit for the inconclusive cases is that the probability that $\operatorname{rank}(W) < n-1$ is fairly large for small $\lambda$.\footnote{
    This probability is certainly bounded away from zero as $n\to \infty$ for fixed $\lambda$ since $\bbP(\operatorname{rank}(W) < n-1) \geq \mathbb{P}(\operatorname{rank}(A) < n-2)$ and the adjacency matrix of the $2$-core is rank degenerate by $3$ with non-vanishing probability by \cite{glasgow2025exact}.
    This coarse bound proves that the probability of non-controllability is nonzero, but does not determine what the exact probability should be since it can also occur that $\operatorname{rank}(W)< n-1$ when $\operatorname{rank}(A) =n$.} 
    For instance, we find that this probability is on the order of $50\%$ when $\lambda = 1.2$.
    Given the philosophy in random matrix theory that singularity should typically arise from local obstructions \cite{tikhomirov2020singularity,costello2008rank,costello2010rank,addario2014hitting,ferber2023singularity,glasgow2025exact}, and that our \Cref{thm: NoSwitching} shows that local obstructions should not cause cospectrality, it seems plausible that most of the graphs in \Cref{tab:2coreexperiments} with $\operatorname{rank}(W)< n-1$ may also be determined by their $\bbR$-spectrum.
\end{description}
In summary, the data where the algorithms are conclusive gives strong evidence that the $2$-core should typically be determined by its $\bbR$-spectrum when it is controllable or almost controllable, and it seems plausible that the non-controllable cases may also often be determined by $\bbR$-spectrum. 
This supports the variant of \Cref{conj: 2Core}.

\begin{table}[htb]
    \centering
    \begin{tabular*}{\textwidth}{@{\extracolsep{\fill}} ccccccc }
        \multirow{2}{*}{$\lambda$} & 
        \multirow{2}{*}{$n$} & 
        Average & 
        Determined & 
        \multirow{2}{*}{$\bbR$-cospectral} & 
        \multicolumn{2}{c}{\emph{\rule[0.6ex]{1.3cm}{0.4pt}\,Inconclusive\,\rule[0.6ex]{1.3cm}{0.4pt}}} \\ 
         & & $\#V(\textnormal{Core}_2))$ & by $\bbR$-spectrum & & $\operatorname{rank}(W)<n-1$ & Terminated 
        \\\noalign{\vspace{3pt}}\hline\hline\noalign{\vspace{1pt}}
        \multirow{3}{*}{1.2}    & 50 & 4.4 & 54 & 0 & 46 & 0\\
                                & 100& 6.9 & 50 & 0 & 49 & 1\\
                                & 200& 10.6 & 49 & 0 & 51 & 0\\
                                & 500& 26.3 & 35 & 0 & 61 & 4\\
                                & 1000 & 49.3 & 36 & 1 & 61 & 2\\
                                & 2000 & 110.0 & 39 & 0 & 52 & 9\\\hline
        \multirow{3}{*}{1.5}    & 50 & 9.8 & 36 & 0 & 63 & 1\\
                                & 100& 21.2 & 49 & 0 & 49 & 2\\
                                & 200& 42.7 & 50 & 1 & 45 & 4\\
                                & 500& 106.9 & 55 & 0 & 39 & 6\\
                                & 1000& 216.8 & 48 & 0 & 44 & 8\\\hline
        \multirow{3}{*}{2}      & 50 & 23.9 & 67 & 1 & 29 & 3\\
                                & 100& 49.0 & 72 & 0 & 14 & 14\\
                                & 200& 94.9 & 75 & 0 & 18 & 7\\
                                & 500& 233.0 & 70 & 0 & 16 & 14\\\hline
        \multirow{3}{*}{2.5}    & 50 & 31.6 & 88 & 1 & 7 & 4\\
                                & 100& 65.6 & 87 & 0 & 5 & 8\\
                                & 200& 132.2 & 89 & 2 & 3 & 6\\
                                & 500& 329.0 &91 & 0 & 6 & 3\\\hline
        \multirow{3}{*}{5}      & 50 & 48.2 & 100 & 0 & 0 & 0\\
                                & 100& 96.2 & 100 & 0 & 0 & 0\\
                                & 200& 191.8 & 99 & 0 & 0 & 1\\
    \end{tabular*}
    \caption{Is the $2$-core of the giant component of $G(n,\lambda/(n-1))$ determined by $\bbR$-spectrum? Tested using Wang and Wang's algorithm \cite{wang2025haemers} on 100 random graphs per $(n,\lambda)$.}
    \label{tab:2coreexperiments}
\end{table}
\subsection{Algorithmic optimizations}\label{sec: AlgoOptimization}

Our code is largely a translation of the Mathematica code from \cite{wang2025haemers} to Sagemath and can be found at \href{https://github.com/nilsvandeberg/cospectrality_sparse_graphs}{\nolinkurl{https://github.com/nilsvandeberg/cospectrality_spa}}\hfill\break\href{https://github.com/nilsvandeberg/cospectrality_sparse_graphs}{\nolinkurl{rse_graphs}}. 
We made two noteworthy changes to enable computation for larger graphs. 

First, the algorithm requires both the elementary divisors $d_1,\ldots, d_n$ of a modified walk matrix $W^p$ and the invertible matrices $U,V \in \operatorname{GL}_{n}(\bbZ)$ that define the similarity between $W^p$ and its Smith normal form. 
The computation of $V$ is expensive because the entries in the walk matrices become extremely large. 
However, one needs to know $V$ only up to a prime power depending on $d_n$; see \cite[Corollary 3]{wang2025haemers}. 
Hence we need only consider $W^p$ modulo this prime power, which can sharply reduce the values in the walk matrix and hence improve the computation time. 

Second, for almost controllable graphs the algorithm replaces the final column of the walk matrix to find an invertible matrix. 
The implementation in \cite{wang2025haemers} determines this column by computing all cofactors of the original final column of $W$. 
The resulting vector is in the kernel of $W$, see \eg \cite[Proposition~1]{wang2021almost}. 
Hence we can also find a vector in the kernel and scale it by computing only one of the cofactors to get the correct final column. 

Even with these adjustments it is still possible for the computation to overflow or take a very long time. 
We have followed the approach in \cite{wang2025haemers} and terminated the algorithm if certain values became too big. 
The cases where this happened are logged in the `Terminated' column of \Cref{tab:2coreexperiments}.

\subsection*{Acknowledgements}
The authors thank Aida Abiad and Emanuel Juliano for discussions related to the topic of this paper. 

Nils Van de Berg is funded by NWO (Dutch Research Council) through the grant  VI.Vidi.213.085. Alexander Van Werde is funded by the Deutsche Forschungsgemeinschaft (DFG, German Research Foundation) under Germany's Excellence Strategy EXC 2044/2 –390685587, Mathematics M\"unster: Dynamics–Geometry–Structure.

\end{document}